\documentclass{article}
\usepackage{upgreek}
\usepackage{latexsym,epsfig,bm,amssymb}
\usepackage{color}
\usepackage{mathptmx,amsthm,mathrsfs}
\input ot1ptm.fd \input ot1ztmcm.fd \input omlztmcm.fd
\input omsztmcm.fd  \input omxztmcm.fd \input ulasy.fd \input ursfs.fd

\def\version{}

\newcommand{\notyet}[1]{}

\DeclareSymbolFont{AMSb}{U}{msb}{m}{n}
\DeclareSymbolFontAlphabet{\mathbb}{AMSb}

\newcommand{\at}[1]{\vert\sb{\sb{#1}}}

\newcommand{\R}{{\mathbb R}}

\newcommand{\Norm}[1]{\left\Vert #1 \right\Vert}
\newcommand{\norm}[1]{\Vert #1 \Vert}
\newcommand{\const}{{\rm const}}

\providecommand{\ltor}[1]{
\ifnum #1=1{\it i}\else\ifnum #1=2{\it ii}\else\ifnum #1=3{\it iii}
\else\ifnum #1=4 {\it iv}\fi\fi\fi\fi
}

\DeclareMathSymbol{\varPhi}{\mathord}{letters}{"08}
\DeclareMathSymbol{\varOmega}{\mathord}{letters}{"0A}

\def\o{\mathaccent"7017}
\newcommand\Ho{\o{H}}

\font\thf cmssdc10 at 11pt

\theoremstyle{plain}
\newtheorem{theorem}{\thf Theorem}[section]

\newtheorem{lemma}[theorem]{\thf Lemma}

\newtheorem{proposition}[theorem]{\thf Proposition}

\theoremstyle{definition}
\newtheorem{definition}[theorem]{Definition}

\theoremstyle{remark}

\makeatletter\@addtoreset{equation}{section}
\makeatother
\renewcommand{\theequation}{\thesection.\arabic{equation}}

\begin{document}

\title{
On global attraction to stationary states for
\\
wave equations with concentrated nonlinearities
}

\author{
{\sc Elena Kopylova}
\footnote{Research supported by the Austrian Science Fund (FWF) under Grant No.\ P27492-N25
and RFBR grant No.\ 16-01-00100}
\\
{\it\small Faculty of Mathematics of Vienna University}\\
{\it\small and  Institute for Information Transmission Problems Russian Academy of Sciences}}

\date{\version}

\maketitle

\begin{abstract}
The global {\it attraction} to stationary states is established for solutions to  3D wave equations with concentrated nonlinearities: 
each finite energy solution converges as $t\to\pm\infty$ to stationary states.
The attraction is caused by nonlinear energy radiation.
\end{abstract}

\section{Introduction}
\label{int-results}
The paper concerns a nonlinear interaction of the {\it real wave field} with a
point oscillator. The system is governed by the following equations
\begin{equation}\label{iKG}
\left\{\begin{array}{c}
\ddot \psi(x,t)=\Delta\psi(x,t)+\zeta(t)\delta(x)\\\\
\lim\limits_{x\to 0}(\psi(x,t)-\zeta(t)G(x))=F(\zeta(t))
\end{array}\right|\quad x\in\R^3,\quad t\in\R,
\end{equation} 
where $G$ is the Green's function of operator $-\Delta$ in $\R^3$, i.e.
\[
  G(x)=\frac{1}{4\pi|x|},
\]
All derivatives here and below  are understood in the sense of distributions.
The nonlinearity  admits a potential
\begin{equation}\label{FU}
 F(\zeta)=U'(\zeta), \quad\zeta\in\R,\quad U\in C^2(\R).
\end{equation} 
We assume that 
\begin{equation}\label{bound-below}
U(\zeta)\to\infty, \quad \zeta\to\pm\infty.
\end{equation}
Furthermore, we assume that the set $Q=\{q\in R: F(q)=0\}$ is nonempty.
Then the system (\ref{iKG}) admits stationary solutions 
$qG(x)$, where $q\in Q$.
We suppose that the set $Q$ satisfies the following condition
\begin{equation}\label{ab}
[a,b]\not\subset Q~~{\rm for}~ {\rm any}~~ a<b.
\end{equation}
Let  $\Ho^1(\R^3)$ be the completion of the  space $C_0^\infty(\R^3)$ in the norm $\Vert\nabla\psi(x)\Vert_{L^2(\R^3)}$.
Equivalently, using Sobolev's embedding theorem, 
$\Ho^1(\R^3)=\{f \in L^6(R^3) : |\nabla f|\in L^2(\R^3)\}$, and
\begin{equation}\label{sob}
\Vert f\Vert_{L^6(\R^3)}\le C\Vert\nabla  f\Vert_{L^2(\R^3)}.
\end{equation}
Denote
\[
\Ho^2(\R^3):=\{ f\in \Ho^1(\R^3),~~\Delta f\in L^2(\R^3)\},\quad t\in\R.
\]
We consider Cauchy problem for system (\ref{iKG}) with initial data 
$\Psi(x,0)=(\psi(x,0),\dot\psi(x,0))$ which  can be represented as the sum of
{\it regular}  component from $\Ho^2(\R^3)\oplus \Ho^1(\R^3)$ 
and  {\it singular} component proportional to $G(x)$ (see Definition \ref{cDdef}). 
Our main goal is the global attraction of the solution $\Psi(x,t)=(\psi(x,t),\dot\psi(x,t))$
to  stationary states:
\[
\Psi(x,t)\to (q_{\pm}G(x),\,0),\quad t\to\pm\infty,\quad q_{\pm}\in Q,
\]
where the asymptotics hold in local $L^2\oplus L^2$-seminorms.

Similar global attraction was established  for the first time 
i) in \cite{Kom95, Kom99, KK2007}
for 1D wave and Klein-Gordon equations coupled to nonlinear oscillators, ii) 
 in \cite{KK2009, KK2010b}
for nD Klein-Gordon and Dirac equations 
with mean field  interaction, and iii) in 
\cite{C2013} for discrete in space and  time nD Klein-Gordon equation  equations  
interacting with a nonlinear oscillator.

In the context of the Schr\"odinger and wave equations
the point interaction of type (\ref{iKG})
was introduced in \cite{AAFT, AAFT1,  AGHH, KP, NP},
where the well-posedness of the Cauchy problem and the blow up solutions  were studied. 
The orbital and asymptotic stability of soliton solutions for the Schr\"odinger equation
with the point interaction has been established in \cite{ANO}. 
The global attraction  for 3D equations 
with the point interaction was  not  studied up to now.
In the present paper
we prove for the first time the global attraction in the case of 3D  wave equation.

Let us comment on our approach. 
First, similarly to \cite{KK2007, KK2009, KK2010b},  we represent the solution as the sum
of {\it dispersive} and {\it singular} components. The dispersive component is a solution
of the free wave equation with the same initial data $\Psi(x,0)$. The singular component  is a solution
of a coupled  system of wave equation with zero initial data and a point source, and of a nonlinear ODE.

We prove the long-time decay of the dispersive component in local $H^2\oplus H^1$-seminorms. 
To establish the decay for regular part of the dispersive  component, corresponding
to regular initial data from $H^2\oplus H^1$, we apply the strong Huygens principle and the energy
conservation  for  the free wave equation.
For  the  remaining  singular part  we apply the strong Huygens principle.
The dispersive decay is caused by the energy radiation to infinity.

Finally, we  study the nonlinear ODE with a source. We prove that the source decays and 
then the attractor of the ODE coincides with the set of zeros of the nonlinear function $F$,
 i.e. with the set $Q$. This allows us to prove the convergence  of the singular component of the solution
 to one of the stationary solution in local $L^2\oplus L^2$-seminorms.

\section{Main results}
\label{sect-results}
\subsection*{Model}
We fix a nonlinear function $F:\R\to\R$ and define the domain
\begin{equation}\label{q}
D_F=\{\psi\in L^2(\R^3):\psi(x)=\psi_{reg}(x)+\zeta G(x),~~\psi_{reg}\in \Ho^2(\R^3),~~
 \zeta\in\R,~~\lim\limits_{x\to 0}\psi_{reg}(x)=F(\zeta)\}
\end{equation}
which generally is not a linear space. The limit in (\ref{q}) is well defined since
$\Ho^2(\R^3)\subset H^2_{loc}(\R^3)\subset C(\R^3)$ by the Sobolev embedding theorem.

Let $H_F$ be a nonlinear operator on the domain $D_F$ defined by 
\begin{equation}\label{HF}
 H_F \psi=\Delta \psi_{reg},\quad\psi\in D_F.
\end{equation}
The system (\ref{iKG}) for $\psi(t)\in D_F$ reads
\begin{equation}\label{KG}
\ddot \psi(x,t)=H_F \psi(x,t),\quad x\in\R^3,\quad t\in\R.
\end{equation}
Let us introduce the phase space  for equation (\ref{KG}).
Denote the space
\[
\dot D=\{\pi\in L^2(\R^3):\pi(x)=\pi_{reg}(x)+\eta G(x),
~~\pi_{reg}\in \Ho^1(\R^3), ~~\eta\in\R\}.
\]
Obviously, $D_F\subset\dot D$.
\begin{definition}\label{cDdef}
${\cal D}_F$ is the Hilbert space of the states 
$\Psi=(\psi(x),\pi(x))\in D_F\oplus\dot D$ equipped with the finite norm
\[
\Vert\Psi\Vert_{\cal D}^2:=\norm{\nabla\psi_{reg}}_{L^2(\R^3)}^2+\norm{\Delta\psi_{reg}}_{L^2(\R^3)}^2
+\norm{\nabla\pi_{reg}}_{L^2(\R^3)}^2+|\zeta|^2+|\eta|^2.
\]
\end{definition}
\subsection*{Well-posedness}
\begin{theorem}\label{theorem-well-posedness}
Let conditions (\ref{FU}) and (\ref{bound-below}) hold. 
Then 
\begin{enumerate}
\item
For every initial data $\Psi(0)=(\psi(0),\dot\psi(0))\in {\cal D}_F$  the equation
(\ref{KG}) has a unique strong solution $\psi(t)$ such that 
\item
The energy is conserved:
\[
{\cal H}_F(\Psi(t)):=
\frac 12 \Big(\Vert\dot\psi(t)\Vert^2_{L^2(\R^3)}
+\Vert\nabla\psi_{reg}(t)\Vert^2_{L^2(\R^3)}\Big)+U(\zeta(t))=\const, \quad t\in\R.
\]
\item
The following a priori bound holds
\begin{equation}\label{zeta-bound}
|\zeta(t)|\le C(\Psi(0)), \quad t\in R. 
\end{equation}
\end{enumerate}
\end{theorem}
This result is proved in \cite[Theorem 3.1]{NP}. For the convenience of readers, we sketch main steps of the proof  in  Appendix  
 in the case $t\ge 0$ clarifying some details  of \cite{NP}. As the result 
 the solution $\psi(x,t)$ to (\ref{KG}) with initial data $\psi(0)=\psi_0\in D_F$, 
 $\dot\psi(0)=\pi_0\in\dot D$ can be represented as the sum
\begin{equation}\label{sol_sum}
\psi(x,t):= \psi_f(x,t)+\psi_S(x,t), \quad t\ge 0,
\end{equation}
where the {\it dispersive component}
$\psi_f(x,t)$ is a unique solution of the Cauchy problem for the free wave equation
\begin{equation}\label{CP1}
\ddot{\psi}_f(x,t) = \Delta\psi_f(x,t),
\quad \psi_f(x,0) = \psi_0(x),\quad\dot\psi_f(x,0)  =  \pi_0(x),
\end{equation}
and the {\it singular component} $\psi_S(x,t)$ is a unique solution of the Cauchy problem 
for the  wave equation with a point source
\begin{equation}\label{CP2}
\ddot\psi_S(x,t)= \Delta\psi_S(x,t) +\zeta(t)\delta(x),
\quad \psi_S(x,0) = 0,\quad\dot\psi_S(x,0)=0.
\end{equation}
Here $\zeta(t)\in C^1_b([0,\infty))$  is a unique solution to the Cauchy problem 
for the following first-order nonlinear ODE
\begin{equation}\label{delay}
\frac {1}{4\pi}\dot\zeta(t)+ F(\zeta(t))=\lambda(t),
\quad \zeta(0)=\zeta_0,
\end{equation}
where 
\begin{equation}\label{lam}
\lambda(t):=\lim\limits_{x\to 0}\psi_f(x,t),\quad t>0,
\end{equation}
Next lemma implies that  limit (\ref{lam}) is well defined,  and there exists $\lambda(0+)=\lim\limits_{t\to 0+}\lambda(t)$. 
 \begin{lemma}\label{wdl}
 Let $(\psi_0,\pi_0)\in {\cal D}_F$. Then\\
i) There exists a unique solution $\psi_f\in C([0;\infty), L^2_{loc})$ to (\ref{CP1}).\\
ii) The limit in (\ref{lam}) exists and is continuous in $t\in [0,\infty)$.\\
iii)  $\dot\lambda\in L^2_{loc}([0,\infty))$.
 \end{lemma}
 \begin{proof}
{\it i)} We split $\psi_f(x,t)$ as 
 \[
 \psi_f(x,t)=\psi_{f,reg}(x,t)+g(x,t),
 \]
 where $\psi_{f,reg}$ and $g$ are the solutions to the free wave equation with  initial data 
 $(\psi_{0,reg},\pi_{0,reg})\in \Ho^2(\R^3)\oplus\Ho^1(\R^3)$ and $(\zeta_0 G,\dot\zeta_0 G)$, respectively.
 By the energy conservation $\psi_{f,reg}\in C([0,\infty),\Ho^2(\R^3))$.  
 Now we obtain  an explicit formula for $g(x,t)$.  Note that 
 $h(x,t)=g(x,t)-\xi(t)G(x)$, where $\xi(t)=\zeta_0+t\dot\zeta_0$, satisfies
 \begin{equation}\label{Smir}
 \ddot h(x,t)=\Delta h(x,t)-\xi(t)\delta(x)
 \end{equation} 
 with zero initial data.
 The unique solution to (\ref{Smir}) is the  spherical wave 
\begin{equation}\label{Smir1}
h(x,t)=-\frac{\theta(t-|x|)}{4\pi|x|}\xi(t-|x|),\quad t\ge 0,
\end{equation}
where $\theta$  is the Heaviside function.
This is  well-known  formula \cite[Section 175]{S} for the retarded potential of the  point particle.
Hence,
\[
g(x,t)=h(x,t)+\xi(t)G(x)=-\frac{\theta(t-|x|)(\zeta_0+(t-|x|)\dot\zeta_0)}{4\pi|x|}
+\frac{\zeta_0+t\dot\zeta_0}{4\pi|x|}\in C([0,\infty),L^2_{loc}(\R^3)).
\]
ii)
We have
\begin{equation}\label{g-lim}
\lim\limits_{x\to 0}g(x,t)=\dot\zeta_0/(4\pi),\quad t>0.
\end{equation}
Moreover,  for any $t\ge 0$
 the  $\lim\limits_{x\to 0}\psi_{f,reg}(x,t)$ exists  because $\Ho^2(\R^3)\subset C(\R^3)$.\\
{\it iii)} Due to (\ref{g-lim}) it remains  to show that
$\dot\psi_{f,reg}(0,t)\in L^2_{loc}([0,\infty))$. This follows immediately from 
\cite[Lemma 3.4]{NP}.
 \end{proof}
\subsection*{Stationary solutions and the main theorem}
The stationary solutions of equation  (\ref{KG}) are solutions of the form
\begin{equation}\label{sol1}
\psi_q(x)=qG(x)\in L^2_{loc}(\R^3),\quad q\in\R.
\end{equation}
\begin{lemma}\label{sol-ex} (Existence of stationary solutions).
Function (\ref{sol1}) is a stationary soliton to (\ref{KG}) if and only if
\begin{equation}\label{qsol}
F(q)=0.
\end{equation}
\end{lemma}
\begin{proof}
Evidently,  $\psi_q(x)$ admits the splitting $\psi_q(x)=\psi_{reg}(x,t)+\zeta(t)G(x)$, where
$\psi_{reg}(x,t)\equiv 0$ and  $\zeta(t)\equiv q$.
Hence, the second equation of (\ref{iKG})  is equivalent to (\ref{qsol}).
\end{proof}
Our main result is the following theorem.
\begin{theorem}[Main Theorem]
\label{main-theorem}
Let assumptions (\ref{FU}), (\ref{bound-below}) and (\ref{ab}) hold and let $\psi(x,t)$ be a solution
to equation (\ref{KG}) with  initial data $\Psi(0)=(\psi(0),\,\dot\psi(0))\in {\cal D}_F$. Then 
\[
(\psi(t),\,\dot\psi(t))\to (\psi_{q_{\pm}},\,0),\quad t\to\pm\infty,\quad q_{\pm}\in Q,
\]
where the  convergence hold in $L^2_{loc}(\R^3)\oplus L^2_{loc}(\R^3)$.
\end{theorem}
It suffices to prove Theorem~\ref{main-theorem} for $t\to+\infty$.
\section{Dispersion component}
\label{sect-splitting}
We will only consider the solution $\psi(x,t)$ restricted to $t\ge 0$.
In this section we extract  regular and singular parts from the  dispersion component  $\psi_f(x,t)$
and establish their  local decay.
First, we represent the initial data  $(\psi(0),\,\dot\psi(0))=(\psi_0,\pi_0)\in {\cal D}_F$ as
\[
(\psi_0,~\pi_0)=(\psi_{0,reg},~\pi_{0,reg})+(\zeta_0G, ~\dot\zeta_0 G)
=(\varphi_{0},~\eta_{0})+(\zeta_0\chi G, ~\dot\zeta_0 \chi G),
\]
where a cut-of function $\chi \in C_0^\infty(\R^3)$ satisfies
\begin{equation}\label{G1}
\chi (x)=\left\{\begin{array}{ll} 1,\quad |x|\le1\\
0,\quad |x|\ge 2
\end{array}\right.
\end{equation}
Let us show that
\begin{equation}\label{vpH} 
(\varphi_{0},~\eta_{0})\in H^2(\R^3)\oplus H^1(\R^3). 
\end{equation} 
Indeed,
\[
(\varphi_{0},~\eta_{0})=(\psi_0-\zeta_0\chi G,~\pi_0-\dot\zeta_0\chi G)
\in L^2(\R^3)\oplus L^2(\R^3),
\]
On the other hand,
\[
(\varphi_{0},~\eta_{0})=(\psi_{0,reg}+\zeta_0(1-\chi)G,~\pi_{0,reg}+\dot\zeta_0(1-\chi)G)
\in \Ho^2(\R^3)\oplus\Ho^1(\R^3).
\]
Now we split the dispersion component $\psi_f(x,t)$  as
\begin{equation}\label{psi-split}
\psi_f(x,t)=\varphi(x,t)+\psi_{G}(x,t),
\quad t\ge 0,
\end{equation}
where $\varphi$ and  $\psi_{G}$  are defined
as solutions to the following Cauchy problems:
\begin{eqnarray}
&&
\ddot\varphi(x,t)=\Delta\varphi(x,t),
\qquad (\varphi,\dot\varphi)\at{t=0}=(\varphi_{0},~\eta_{0}),
\label{KG-cp-1}
\\
\nonumber
\\
&&
\ddot\psi_{G}(x,t)=\Delta\psi_{G}(x,t),
\qquad (\psi_{G},\dot\psi_{G})\at{t=0}=(\zeta_0\chi G, \dot\zeta_0\chi G),
\label{KG-cp}
\end{eqnarray}
and study the decay properties of  $\psi_{G}$ and $\varphi$.
\begin{lemma}\label{lemma-decay-G1}
For the  solution $\psi_{G}(x,t)$ to (\ref{KG-cp}) the strong Huygens principle holds:
\begin{equation}\label{psi2-decay}
\psi_{G}(x,t)=0 ~~{\rm for}~~t\ge |x|+2.
\end{equation}
\end{lemma}
\begin{proof}
The solution $\varphi_G(x,t)$ to the free wave equation with initial data 
$(0, \chi G)\in H^1(\R^3)\oplus L^2(\R^3)$ satisfies the strong Huygens principle  due  to \cite[Theorem XI.87]{RS3}.
Further,
$$
\psi_G(x,t)=\zeta_0\dot\varphi_G(x,t)+\dot\zeta_0\varphi_G(x,t).
$$
Then (\ref{psi2-decay}) follows.
\end{proof}
The following lemma states a local decay of solutions to the free wave equation 
with regular initial data from $H^2(\R^3)\oplus H^1(\R^3)$.
\begin{lemma}\label{lemma-decay-psi1}
Let $\varphi(t)$ be a solution to (\ref{KG-cp-1}) 
with  initial data $\phi_0=(\varphi_0,\eta_0)\in H^2(\R^3)\oplus H^1(\R^3)$. Then
\begin{equation}\label{en-dec}
\Norm{(\varphi(t),\dot \varphi(t))}_{H^2(B_R)\oplus H^1(B_R)}\to 0,\quad t\to\infty,\quad\forall R>0,
\end{equation} 
where $B_R$ is the ball of radius $R$.
\end{lemma}
\begin{proof}
For any $r\ge 1$ denote  $\chi_r=\chi(x/r)$, where $\chi(x)$ is  a cut-off function defined in (\ref{G1}).
Let $u_r(t)$ and $v_r(t)$  be the solutions to the free wave
equations with the initial data $\chi_r \phi_0$
and $(1-\chi_r) \phi_0$, respectively, so that
$u(t)=u_r(t)+v_r(t)$. By  the strong Huygens principle 
\[
u_r(x,t)=0 ~~{\rm for}~~t\ge |x|+2r.
\]
To conclude (\ref{en-dec}), it remains to note that
\begin{eqnarray}\nonumber
\Vert(v_r(t),\dot v_r(t))\Vert_{H^2(B_R)\oplus H^1(B_R)}&\le& C(R)
\Vert(v_r(t),\dot v_r(t))\Vert_{\Ho^2(\R^3)\oplus H^1(\R^3)}
= C(R)\Vert (1-\chi_r) \phi_0\Vert_{\Ho^2(\R^3)\oplus H^1(\R^3)}\\
\label{en-dec1}
&\le& C(R)\Vert (1-\chi_r) \phi_0\Vert_{H^2(\R^3)\oplus H^1(\R^3)}
\end{eqnarray}
due to the energy conservation for the free wave equation. We also use the embedding
$\Ho^1(\R^3)\subset L^6(\R^3)$. The right-hand side of (\ref{en-dec1})
could be made arbitrarily small if $r\ge 1$ is  sufficiently large.
\end{proof}
Finally, (\ref{psi-split}) , (\ref{psi2-decay}) , (\ref{vpH}) and Lemma \ref {lemma-decay-psi1} imply
\begin{equation}\label{psif-dec}
\Norm{(\psi_f(t),\dot\psi_f(t))}_{H^2(B_R)\oplus H^1(B_R)}\to 0,\quad t\to\infty,\quad\forall R>0.
\end{equation}
\section{Singular component}
\label{sect-spectral}
Due to (\ref{psif-dec}) to prove Theorem~\ref{main-theorem}  it suffices
to deduce the convergence to stationary states for the singular component $\psi_S(x,t)$ of the solution.
\begin{proposition}\label{propfin}
Let assumptions of Theorem \ref{main-theorem} hold, and let $\psi_S(t)$ be 
a solution to  (\ref{CP2}). Then
\[
(\psi_S(t),\dot\psi_S(t))\to (\psi_{q_{\pm}},~0),\quad t\to\infty,
\]
where the convergence holds in $L^2_{loc}(\R^3)\oplus L^2_{loc}(\R^3)$.
\end{proposition}
\begin{proof}
The unique solution to (\ref{CP2}) is the  spherical wave 
\begin{equation}\label{psiSf}
\psi_S(x,t)=\frac{\theta(t-|x|)}{4\pi|x|}\zeta(t-|x|),\quad t\ge 0,
\end{equation}
cf. (\ref{Smir})-(\ref{Smir1}).
Then a priori bound (\ref{zeta-bound}) and equation (\ref{delay}) imply that 
\[
(\psi_S(t),\dot\psi_S(t))\in   L^2(B_R)\oplus L^2(B_R),\quad 0\le  R<t.
\]
First, we obtain  a convergence of $\zeta(t)$.
\begin{lemma}\label{PC2}
There exists the limit
 \begin{equation}\label{zetalim}
\zeta(t)\to q_+,\quad t\to\infty,
\end{equation}
where $q_+\in Q$.
\end{lemma}
\begin{proof}
From (\ref{zeta-bound}) it follows that $\zeta(t)$ has the upper and lower limits:
\[
\underline{\lim}_{t\to\infty}\zeta(t)=a,\quad
\overline{\lim}_{t\to\infty}\zeta(t)=b.
\]
Suppose  that $a<b$.
Then the  trajectory $\zeta(t)$ oscillates between $a$ and $b$.
Assumption (\ref{ab}) implies that   $F(\zeta_0)\not =0$ for some $\zeta_0\in (a,b)$. 
For the concreteness, let us assume that $F(\zeta_{0})>0$.
The convergence (\ref{psif-dec}) implies that
\begin{equation}\label{lam-dec}
\lambda(t)=\psi_f(0,t)\to 0,\qquad t\to\infty.
\end{equation}
Hence, for  sufficiently large $T$ we have
\[
-F(\zeta_{0})+\lambda(t)<0, \quad t\ge T.
\]
Then for $t\ge T$ the transition of the trajectory from left to right through the point $\zeta_0$ is impossible by (\ref{delay}).
Therefore, $a=b=q_+$. Finally  $F(q_+)=0$ by (\ref{delay}).
\end{proof}
Further, 
\begin{equation}\label{tetalim}
\theta(t-|x|)\to 1, \quad t\to\infty
\end{equation}
uniformly in  $|x|\le R$.
Then (\ref{psiSf}) and (\ref{zetalim})  imply that
\[
\psi_S(t)\to q_+G, \quad t\to\infty,
\]
where the convergence holds in $L^2_{loc}(\R^3)$.
It remains to deduce the convergence of $\dot\psi_S(t)$.
We have
\[
\dot\psi_S(x,t)=\frac{\theta(t-|x|)}{4\pi|x|}\dot\zeta(t-|x|), \quad t>|x|.
\]
From  (\ref{zetalim}),  (\ref{delay}) and (\ref{lam-dec})  it follows  that
$\dot\zeta(t)\to 0$ as $ t\to\infty$.
Then
$$
\dot\psi_S(t)\to 0,\qquad t\to\infty
$$
in $L^2_{loc}(\R^3)$ by (\ref{tetalim}).
This completes the proof of  Proposition \ref{propfin} and 
Theorem~\ref{main-theorem}.
\end{proof}
\appendix
\setcounter{equation}{0}
\protect\renewcommand{\thesection}{\Alph{section}}
\protect\renewcommand{\theequation}{\thesection.\arabic{equation}}
\protect\renewcommand{\thesubsection}{\thesection.\arabic{subsection}}
\protect\renewcommand{\thetheorem}{\Alph{section}.\arabic{theorem}}
\setcounter{equation}{0}
\section{Appendix}
\label{nonlin-sect}
Here we sketch main steps of the proof \cite[Theorem 3.1]{NP}. 
First we adjust the nonlinearity $F$ so that it becomes Lipschitz-continuous.
Define
\begin{equation}\label{Lambda}
\Lambda(\Psi_0)=\sup\{|\zeta|: \zeta\in\R,\, U(\zeta)\le H_F(\Psi_0)\},
\end{equation}
where $\Psi_0=\Psi(0)\in {\cal D}_F$ is the initial data from Theorem \ref{theorem-well-posedness}. 
Then we may pick a modified potential function
$\tilde U(\zeta)\in C^2(\R)$, so that
\begin{equation}\label{Lambda1}
\left\{\begin{array}{ll}
\tilde U(\zeta)= U(\zeta),\quad |\zeta|\le\Lambda(\Psi_0)\\\\
\tilde U(\zeta)>H_F(\Psi_0),\quad |\zeta|>\Lambda(\Psi_0),
\end{array}\right.
\end{equation}
and the function $\tilde F(\zeta)=\tilde U'(\zeta)$ is Lipschitz continuous:
\begin{equation}\label{Lambda22}
|\tilde F(\zeta_1)-\tilde F(\zeta_2)|\le C|\zeta_1-\zeta_2|,\quad\zeta_1,\zeta_2\in\R.
\end{equation}
We consider the Cauchy problem for (\ref{KG})) with the modified nonlinearity $\tilde F$.
According to Lemma \ref{wdl} there exist  the unique solution 
$\psi_f(x,t)\in C([0,\infty),L^2_{loc}(\R^3))$ to  (\ref{CP1}) and $\lambda(t)=\lim\limits_{x\to 0}\psi_f(x,t)\in C([0,\infty))$.
The following lemma follows by  the contraction mapping principle.
\begin{lemma}\label{LLWP}
Let conditions  (\ref {Lambda1})-(\ref {Lambda22}) be satisfies. 
Then there exists $\tau>0$ such that  the Cauchy problem
\begin{equation}\label{delay1}
\frac {1}{4\pi}\dot\zeta(t)+\tilde F(\zeta(t))=\lambda(t),\quad \zeta(0)=\zeta_{0}
\end{equation}
has a unique solution $\zeta\in C^1([0,\tau])$.
\end{lemma}
Denote 
$$
\psi_S(t,x):=\frac{\theta(t-|x|)}{4\pi|x|}\zeta(t-|x|), \quad t\in [0,\tau],
$$
with $\zeta$ from Lemma \ref{LLWP}.
Now we establish the local well-posedness.
\begin{proposition}\label{TLWP}
Let  the conditions  (\ref {Lambda1})--(\ref {Lambda22}) hold.
Then the function $\psi(x,t):= \psi_f(x,t)+\psi_S(x,t)$ 
is a unique strong  solution to the system
\begin{equation}\label{CP}
\left\{\begin{array}{c}
\ddot \psi(x,t)=\Delta\psi(x,t)+\zeta(t)\delta(x)\\\\
\lim\limits_{x\to 0}(\psi(x,t)-\zeta(t)G(x))=\tilde F(\zeta(t))
\end{array}\right|\quad x\in\R^3,\quad t\in [0,\tau].
\end{equation}
 with initial data  
\[
\psi(0) = \psi_0\in D_{\tilde F},\quad\dot\psi(0)  =  \pi_0\in\dot D,
\]
and satisfies
\begin{equation}\label{CDD}
(\psi(t),\dot\psi(t))\in {\cal D}_{\tilde F},\quad t\in [0,\tau].
\end{equation}
\end{proposition}
\begin{proof}
Since $\zeta(t)$ solves (\ref{delay1}) one has 
\begin{equation}\label{lim_zeta}
\lim_{x \to 0}\, ( \psi(t,x)\!-\!\zeta(t) G(x))=\lambda(t)+
\lim_{x\to 0}\Big(\frac{\theta(t-|x|)\zeta(t-|x|)}{4\pi|x|}-\frac{\zeta(t)}{4\pi|x|}\Big)
=\lambda(t)-\frac{1}{4\pi}\dot\zeta(t)=\tilde F(\zeta(t)).
\end{equation}
Therefore, the second equation of (\ref{CP}) is satisfied.  
Further,
$$
\ddot\psi=\ddot\psi_f+\ddot\psi_S=\Delta\psi_f
+\Delta\psi_S+\zeta\delta=\Delta\psi+\zeta\delta
$$
and $\psi$ solves the first equation of (\ref{CP}) then.  Let us check (\ref{CDD}).
Note that  the function  $\psi_{reg,1}(x,t)=\psi(x,t)-\zeta(t) G_1(x)$, where $G_1(x)=G(x)e^{-|x|}$,
is a solution to
\[
\ddot\psi_{reg,1}(x,t)=\Delta\psi_{reg,1}(x,t)+(\zeta(t)-\ddot\zeta(t)) G_1(x)
\]
with  initial data from $H^2\oplus H^1$. Lemma \ref{wdl}-iii) and equation (\ref{delay1}) imply that
$\ddot\zeta\in L^2([0,\tau])$. Hence, 
\[
(\psi_{reg,1}(x,t),\dot\psi_{reg,1}(x,t))\in  H^2\oplus H^1,\quad t\in [0,\tau] 
\] 
by \cite[Lemma 3.2]{NP}. Therefore,
\[
\psi_{reg}(x,t)=\psi(x,t)-\zeta(t) G(x)=\psi_{reg,1}(x,t)+\zeta(t)(G_1(x)-G(x))
\]
satisfies
$(\psi_{reg}(t),\dot\psi_{reg}(t))\in \Ho^2(\R^3)\oplus\Ho^1(\R^3)$, $t\in [0,\tau]$,
and (\ref{CDD}) holds then.

Suppose now that $\tilde\psi=\tilde\psi_{reg}+\tilde\zeta G$, such that 
$(\tilde\psi,\dot{\tilde\psi)}\in {\cal D}_{\tilde F}$,
 is another strong solution of (\ref{CP}). 
Then, by reversing the above argument, the second equation of (\ref{CP})
 implies that $\tilde\zeta$ solves the Cauchy problem
(\ref{delay1}). The uniqueness of the solution of (\ref{delay1}) implies that
$\tilde\zeta=\zeta$. Then, defining
$$
\psi_S(t,x):=\frac{\theta(t-|x|)}{4\pi|x|}\zeta(t-|x|), \quad t\in [0,\tau],
$$
for $\tilde\psi_f=\tilde\psi-\psi_S$ one obtains
$$
\ddot{\tilde\psi}_f=\ddot{\tilde\psi}-\ddot\psi_S=\Delta\tilde\psi_{reg}-(\Delta\psi_S+\zeta\delta)=
\Delta(\tilde\psi_{reg}-(\psi_S-\zeta G))=\Delta\tilde\psi_f\,,
$$
i.e $\tilde\psi_f$ solves the Cauchy problem (\ref{CP1}). 
Hence, $\tilde\psi_f=\psi_f$ by the uniqueness of the solution to
(\ref{CP1}), and then $\tilde\psi=\psi$.
\end{proof}
According to \cite[Lemma 3.7]{NP}
\begin{equation}\label{cHFT}
{\cal H}_{\tilde F}(\Psi(t))=\Vert\dot\psi(t)\Vert^2+\Vert\nabla\psi_{reg}(t)\Vert^2
+\tilde U(\zeta(t))=const,\quad t\in [0,\tau].
\end{equation}
\begin{lemma}\label{cor1}
The following identity holds
\begin{equation}\label{UtU}
\tilde U(\zeta(t))=U(\zeta(t)), \quad t\in [0,\tau].
\end{equation}
\end{lemma}
\begin{proof}
First note that
\[
{\cal H}_{F}(\Psi_0)\ge  U(\zeta_{0}).
\]
Therefore, $|\zeta_0|\le\Lambda(\Psi_0)$, and then  $\tilde U(\zeta_0)=U(\zeta_0)$, 
${\cal H}_{\tilde F}(\Psi_0)={\cal H}_{F}(\Psi_0)$.
Further, 
\[
{\cal H}_{F}(\Psi_0)={\cal H}_{\tilde F}(\Psi(t))\ge \tilde U(\zeta(t)),\quad t\in [0,\tau].
\]
Hence (\ref{Lambda1}) implies that
\begin{equation}\label{zeta_bound}
|\zeta(t)|\le\Lambda(\Psi_0),\quad t\in [0,\tau].
\end{equation}
\end{proof}

From the identity  (\ref{UtU}) it follows that we can replace $\tilde F$ by $F$ 
in  Proposition \ref{TLWP} and in (\ref{cHFT}).
The solution $\Psi(t)=(\psi(t),\dot\psi(t))\in {\cal D}$ constructed in Proposition  \ref{TLWP}
exists for $0\le t\le\tau$, where the time span $\tau$ in Lemma \ref{LLWP} depends only on $\Lambda(\Psi_0)$.
Hence, the bound (\ref{zeta_bound}) at $t=\tau$ allows us to extend the solution $\Psi$ to the time
interval $[\tau, 2\tau]$. We proceed by induction to obtain the solution for all $t\ge 0$.

\end{document}